\documentclass[a4paper,10pt]{amsart}
\usepackage{amsmath}
\usepackage{cases}

\UseRawInputEncoding

\usepackage{amsfonts}
\usepackage[colorlinks,linkcolor=blue,citecolor=blue]{hyperref}
\usepackage{latexsym, amssymb, amsmath, amsthm, bbm}
\usepackage[all]{xy}
\usepackage{pgfplots}
\usepackage{enumerate}

\DeclareSymbolFont{EulerExtension}{U}{euex}{m}{n}
\DeclareMathSymbol{\euintop}{\mathop} {EulerExtension}{"52}
\DeclareMathSymbol{\euointop}{\mathop} {EulerExtension}{"48}

\allowdisplaybreaks[4]

\def \id{\operatorname{id}}

\def \ord{\operatorname{ord}}

\def \C{\mathcal{C}}

\def \Z{\mathbb{Z}}

\def \k{\Bbbk}

\def \dim{\operatorname{dim}}

\def \Hom{\operatorname{Hom}}

\def \T{\operatorname{T}}

\def \ord{\operatorname{ord}}

\def \C{\mathcal{C}}

\def \X{\mathcal{X}}

\def \S{\mathcal{S}}
\def \Z{\mathbb{Z}}

\def \T{\mathrm{T}}

\def \End{\operatorname{End}}

\def \lcm{\operatorname{lcm}}

\numberwithin{equation}{section}

\newtheorem{theorem}{Theorem}[section]
\newtheorem{lemma}[theorem]{Lemma}
\newtheorem{proposition}[theorem]{Proposition}
\newtheorem{corollary}[theorem]{Corollary}
\newtheorem{definition}[theorem]{Definition}

\newtheorem{remark}[theorem]{Remark}

\newtheorem{notation}[theorem]{Notation}

\begin{document}
\title[Note on Invariance and Finiteness for the Exponent of Hopf algebras]{Note on Invariance and Finiteness for the Exponent of Hopf algebras}


\author[K. Li]{Kangqiao Li}
\address{Department of Mathematics, Nanjing University, Nanjing 210093, China}
\email{kqli@nju.edu.cn}

\date{}

\begin{abstract}
There are two notions of exponent of finite-dimensional Hopf algebras introduced and studied in the literature. In this note, we discuss and compare their properties including invariance and finiteness in this note.
Specifically, one notion is invariant under twisting and taking the Drinfeld double, just like the other one. We also find that if the non-cosemisimplicity and dual Chevalley property hold, both exponents are infinite in characteristic $0$ but finite in positive characteristic.
\end{abstract}

\maketitle

\section{Introduction}

The notions of exponent of Hopf algebras have been introduced and studied since 1999. This process arised from a conjecture by Kashina \cite{Kas99,Kas00} that $n$th Sweedler power $[n]$ is trivial on a semisimple and cosemisimple Hopf algebra of dimension $n$. 
She also verified this property on a number of known examples. One notion of exponents for a Hopf algebra $H$ is considered to be the least positive integer $n$ such that $[n]$ is trivial, which is denoted by $\exp_0(H)$ in this note. Actually it coincides with the exponent of a group $G$ when $H$ is the group algebra $\k G$.

However, the term ``exponent'' of Hopf algebras was firstly introduced by Etingof and Gelaki \cite{EG99}. Their definition of the exponent, denoted by $\exp(H)$, is slightly different from $\exp_0(H)$ above. They provided at first various properties for $\exp(H)$ when $H$ is finite-dimensional, especially the invariance properties. Specifically, this exponent is invariant under the duality, taking the opposite algebra, twisting and taking the Drinfeld double. One more important result in \cite{EG99} is that $\exp(H)$ is finite and divides $\dim(H)^3$, as long as $H$ is semisimple and cosemisimple (and thus involutory). This partially answers Kashina's conjecture, because of an immediate observation that $\exp(H)=\exp_0(H)$ holds when $H$ is involutory (or pivotal). Another result about the finiteness is that $\exp(H)<\infty$ in positive characteristic when $H$ is finite-dimensional.

On the other hand, $\exp_0(H)$ seems different from $\exp(H)$ for general Hopf algebras by definitions. One might ask whether $\exp_0(H)$ has similar properties with $\exp(H)$. Some properties are studied by Landers, Montgomery and Schauenburg \cite{LMS06}. For instance, they showed that $\exp_0(H)$ is also invariant under taking the opposite algebra. As for the finiteness, the author and Zhu \cite{LZ19} found that $\exp_0(H)=\infty$ if $H$ is non-cosemisimple in characteristic $0$ with the dual Chevalley property, and that $\exp_0(H)<\infty$ if $H$ is finite-dimensional and pointed. Moreover, there are other researches involving exponents of Hopf algebras, such as \cite{EG02}, \cite{KSZ06}, \cite{MVW16} and \cite{SV17}, etc.

This note is an attempt to complete comparisons of properties between $\exp_0(H)$ and $\exp(H)$ when $H$ is finite-dimensional. By our results, it might be suggested that both exponents have almost the same properties and even the same values. Our first result is that $\exp_0(H)$ is also invariant under twisting and taking the Drinfeld double, which is shown in Section 2. In order to compare the finiteness properties, we compute to obtain a formula on the exponent of the pivotal semidirect product $H\rtimes\k\langle S^2\rangle$ in Section 3, stating that $\exp(H\rtimes\k\langle S^2\rangle)=\lcm(\exp(H),\exp(H_0))$. Finally in Section 4, we show that if $H$ is non-cosemisimple with the dual Chevalley property, both of its exponents are infinite in characteristic $0$ but finite in positive characteristic.

\section{Exponents and Their Invariance}\label{Section2}

We start by recalling the definitions and basic properties of exponents. Let $(H,m,u,\Delta,\varepsilon)$ be a Hopf algebra with bijective antipode $S$ over a field $\Bbbk$. Sweedler notation $\Delta(h)=\sum h_{(1)}\otimes h_{(2)}$ for $h\in H$ is always used. We also denote following $\Bbbk$-linear maps for convenience:
\begin{eqnarray*}
m_n: & H^{\otimes n}\rightarrow H, & h_1\otimes h_2\otimes \cdots \otimes h_n \mapsto h_1 h_2\cdots h_n, \\
\Delta_n: & H\rightarrow H^{\otimes n}, & h\mapsto \sum h_{(1)}\otimes h_{(2)}\otimes\cdots\otimes h_{(n)}
\end{eqnarray*}
and the \textit{Sweedler power} $[n]:=m_n\circ\Delta_n$ for each positive integer $n$. The two notions of \textit{exponent} (\cite{Kas99,Kas00} and \cite{EG99}) of $H$ are defined respectively as:
\begin{eqnarray*}
\exp(H) &:=& \min\{n\geq 1\mid m_n\circ (\id\otimes S^{-2}\otimes\cdots\otimes S^{-2n+2})\circ \Delta_n=u\circ\varepsilon \},  \\
\exp_0(H) &:=& \min\{n\geq 1\mid [n]:=m_n\circ \Delta_n=u\circ\varepsilon \}.
\end{eqnarray*}
Moreover in this note, we always make following conventions to cover infinite cases:
\begin{itemize}
  \item $\min \varnothing = \infty$;
  \item Each positive integer divides $\infty$, and $\infty$ divides $\infty$;
  \item Any positive integer (or $\infty$) divided by $\infty$ must be $\infty$ itself.
\end{itemize}
It is immediate that whenever finite and infinite, $\exp_0(H)=\exp(H)$ when $H$ is involutory.

Now we list invariance properties for $\exp(H)$ below, which are introduced in \cite{EG99}. Recall that an element $J\in H\otimes H$ is called a \textit{left twist} for $H$ (\cite{Dri86}), if $J$ is invertible satisfying
$$(J\otimes 1)\cdot (\Delta\otimes\id)(J)=(1\otimes J)\cdot (\id\otimes\Delta)(J),$$
A Hopf algebra $(H^J,m,u,\Delta^J,\varepsilon)$ with antipode $S^J$ could be constructed afterwards. Besides, the Drinfeld double of $H$ is denoted by $D(H)$.

\begin{lemma}\label{lem:expbasic}(\cite{EG99})
Let $H$ be a finite-dimensional Hopf algebra. Then
\begin{itemize}
  \item[(1)] If $\exp(H)<\infty$, then $m_n\circ(\id\otimes S^{-2}\otimes\cdots\otimes S^{-2n+2})\circ\Delta_n=u\circ\varepsilon$ if and only if $\exp(H)\mid n$;
  \item[(2)] $\exp(H^\ast)=\exp(H)$;
  \item[(3)] $\exp(H^\mathrm{op})=\exp(H^\mathrm{cop})=\exp(H)$;
  \item[(4)] If $H^\prime$ is a Hopf subalgebra or quotient of $H$, then $\exp(H^\prime)\mid\exp(H)$;
  \item[(5)] Let $H^\prime$ be another Hopf algebra. Then $\exp(H\otimes H^\prime)=\lcm(\exp(H),\exp(H^\prime))$;
  \item[(6)] Let $J$ be a (left or right) twist for $H$. Then $\exp(H^J)=\exp(H)$;
  \item[(7)] $\exp(D(H))=\exp(H)$.
\end{itemize}
\end{lemma}

\begin{remark}
Items (1), (2), (4) and (5) are contained in \cite[Proposition 2.2]{EG99}. Item (3) is \cite[Corollaray 2.6]{EG99}, while (6) and (7) are respectively \cite[Theorem 3.3 and Corollary 3.4]{EG99}.
\end{remark}

\begin{lemma}\label{lem:exp0basic}(\cite{Kas99} and \cite{LMS06})
Let $H$ be a finite-dimensional Hopf algebra. Then
\begin{itemize}
  \item[(1)] If $\exp_0(H)<\infty$, then $[n]:=m_n\circ\Delta_n=u\circ\varepsilon$ if and only if $\exp_0(H)\mid n$;
  \item[(2)] $\exp_0(H^\ast)=\exp_0(H)$;
  \item[(3)] $\exp_0(H^\mathrm{op})=\exp_0(H^\mathrm{cop})=\exp_0(H)$;
  \item[(4)] If $H^\prime$ is a Hopf subalgebra or quotient of $H$, then $\exp_0(H^\prime)\mid\exp_0(H)$;
  \item[(5)] Let $H^\prime$ be another Hopf algebra. Then $\exp_0(H\otimes H^\prime)=\lcm(\exp_0(H),\exp_0(H^\prime))$.
\end{itemize}
\end{lemma}

\begin{remark}
Items (1), (2), (4) and (5) are direct and found in \cite{Kas99}, and (3) is \protect{\cite[Proposition 2.2(2)]{LMS06}}.
\end{remark}

We end this section by verifying that $\exp_0(H)$ is invariant under twisting and taking the Drinfeld double as well.

\begin{proposition}\label{prop:exp0basic2}
Let $H$ be a finite-dimensional Hopf algebra. Then
\begin{itemize}
  \item[(6)] Let $J$ be a (left or right) twist for $H$. Then $\exp_0(H^J)=\exp_0(H)$;
  \item[(7)] $\exp_0(D(H))=\exp_0(H)$.
\end{itemize}
\end{proposition}

\begin{proof}
\begin{enumerate}
  \item[(6)] Let $J=\sum\limits_i J_i\otimes J^i\in H\otimes H$ be a left twist for $H$, and its inverse is denoted by $J^{-1}=\sum\limits_i (J^{-1})_i\otimes (J^{-1})^i$. We remark that the definition of $S^J$ provides
      $$S^J(h)=\sum_{k,l} J_k S(J^k) S(h) S((J^{-1})_l) (J^{-1})^l\;\;\;\;(\forall h\in H).$$

      In order to compute the Sweedler power on $H^J$, we use the equation in \protect{\cite[Lemma 2.5]{KMN12}} that
      $$ h^{[n+1]}=m_{n+2}\circ(\id\otimes\Delta^J_n\otimes\id)\left[(1\otimes J)(1\otimes\Delta(h))(J^{-1}\otimes 1)\right]$$
      for each positive integer $n$. Meanwhile, the equation also holds when $n=0$ with conventions $m_0=u$ and $\Delta^J_0=\varepsilon$.

      Assume $\exp_0(H^J):=N_0<\infty$. It follows that $m_{N_0 -1}\circ\Delta^J_{N_0 -1}=S^J$ on $H^J$, because they are both the convolution inverses of $\id\in \Hom_\Bbbk(H^J,H)$. Then for all $h\in H$, we make computations:
      \begin{eqnarray*}
      h^{[N_0]}
      &=& m_{N_0+1}\circ(\id\otimes\Delta^J_{N_0-1}\otimes \id)\left[(1\otimes J)(1\otimes\Delta(h))(J^{-1}\otimes 1)\right] \\
      &=& m_3\circ(\id\otimes S^J\otimes \id)\left[(1\otimes J)(1\otimes\Delta(h))(J^{-1}\otimes 1)\right] \\
      &=& \sum m_3\circ(\id\otimes S^J\otimes \id)\left[\sum\limits_{i,j} (J^{-1})_j\otimes J_ih_{(1)}(J^{-1})^j\otimes J^i h_{(2)} \right] \\
      &=& \sum\limits_{i,j} (J^{-1})_j S^J\left(J_ih_{(1)}(J^{-1})^j\right) J^i h_{(2)} \\
      &=& \sum\limits_{i,j,k,l} (J^{-1})_j J_k S(J^k) S\left(J_ih_{(1)}(J^{-1})^j\right) S((J^{-1})_l) (J^{-1})^l J^i h_{(2)} \\
      &=& \sum\limits_{i,j,k,l} (J^{-1})_j J_k S\left((J^{-1})_l J_ih_{(1)}(J^{-1})^j J^k\right) (J^{-1})^l J^i h_{(2)}. \\
      \end{eqnarray*}
      Note that
      $$\sum\limits_{j,k}(J^{-1})_j J_k\otimes (J^{-1})^j J^k=\sum\limits_{i,l}(J^{-1})_l J_i\otimes (J^{-1})^l J^i =J^{-1}J=1\otimes 1.$$
      Thus
      \begin{eqnarray*}
      h^{[N_0]}
      &=& \sum\limits_{i,j,k,l} (J^{-1})_j J_k S\left((J^{-1})_l J_ih_{(1)}(J^{-1})^j J^k\right) (J^{-1})^l J^i h_{(2)} \\
      &=& \sum S(h_{1})h_{(2)} ~=~ \varepsilon(h)1.
      \end{eqnarray*}
      That is to say $[N_0]=u\circ\varepsilon$, and thus $\exp_0(H)\mid\exp_0(H^J)$ by Lemma \ref{lem:exp0basic}(1). However, it is known that $J^{-1}$ is a left twist for $H^J$ and $H=(H^J)^{J^{-1}}$. Therefore, $\exp_0(H^J)\mid\exp_0(H)$ holds similarly. As a consequence, we have $\exp_0(H^J)=\exp_0(H)$. Of course, the process above also shows that $\exp_0(H^J)=\infty$ if and only if $\exp_0(H)=\infty$.

      The property holds when $J$ is a right twist as well, since $J^{-1}$ is a left twist for $H$ at that time.

  \item[(7)] As mentioned in \cite[Section 2]{LMS06}, this could be inferred by the following isomorphism between Hopf algebras introduced in \cite{DT94}:
      $$D(H)=(H^{\ast\mathrm{cop}}\otimes H)_\sigma,$$
      where $\sigma: (f\otimes h, f'\otimes h')\mapsto \langle f,1\rangle \langle f',h\rangle \langle \varepsilon,h'\rangle$ is a left 2-cocycle for $H^{\ast\mathrm{cop}}\otimes H$, and $(H^{\ast\mathrm{cop}}\otimes H)_\sigma$ denotes the corresponding 2-cocycle deformation.

      Specifically, according to Lemma \ref{lem:exp0basic} and the duality between (left) 2-cocycles and twists, we could know that
      \begin{eqnarray*}
      \exp_0(D(H))
      &=& \exp_0((H^{\ast\mathrm{cop}}\otimes H)_\sigma)
      ~=~ \exp_0(((H^{\ast \rm cop}\otimes H)_\sigma)^\ast) \\
      &=& \exp_0(((H^{\ast \rm cop}\otimes H)^\ast)^{\sigma^\ast})
      ~=~ \exp_0((H^{\ast \rm cop}\otimes H)^{\ast}) \\
      &=& \exp_0(H^{\ast \rm cop}\otimes H)
      ~=~ {\rm lcm}(\exp_0(H^{\ast \rm cop}),\exp_0(H)) \\
      &=& {\rm lcm}(\exp_0(H^\ast),\exp_0(H))
      ~=~ {\rm lcm}(\exp_0(H),\exp_0(H)) \\
      &=& \exp_0(H),
      \end{eqnarray*}
      where $\sigma^\ast$ denotes the left twist for $H^\ast$ dual to $\sigma$.
\end{enumerate}
\end{proof}

\section{Exponent of the Pivotal Semidirect Product \texorpdfstring{$H\rtimes \Bbbk\langle S^2\rangle$}{}}

In fact, $\exp(H)=\exp_0(H)$ holds as long as $H$ is pivotal. Recall that a Hopf algebra $H$ is said to be \textit{pivotal}, if there exists a grouplike element $g\in H$ such that
$$\forall h\in H,~S^2(h)=ghg^{-1}.$$
Such a grouplike element $g$ is called a \textit{pivotal element} of $H$. The claim above could be followed from the lemma below:

\begin{lemma}(\cite[Lemma 4.2]{Shi15})\label{lem:pivotpow}
Let $H$ be a Hopf algebra with a grouplike element $g\in H$. Denote $\varphi$ as the inner automorphism on $H$ determined by $g$. Then
$$(hg)^{[n]}=\sum h_{(1)}\varphi(h_{(2)})\cdots\varphi^{n-1}(h_{(n)})g^n$$
holds for each $n\geq 1$ and all $h\in H$.
\end{lemma}

\begin{corollary}\label{cor:pivotal}
Let $H$ be a pivotal Hopf algebra. Then $\exp_0(H)=\exp(H)$.
\end{corollary}

A known result is that any finite-dimensional $H$ could be embedded into a pivotal Hopf algebra, namely a semidirect product $H\rtimes\k\langle S^2\rangle$. Thus $\exp(H)$ and $\exp_0(H)$ are bounded by the equal exponents of $H\rtimes\k\langle S^2\rangle$, especially when the latter is finite.

Let us recall the definition. Firstly it is clear from \cite[Theorem 1]{Rad76} that the subgroup generated by $S^2\in\End_\k(H)$ is finite, which is denoted by $\langle S^2\rangle$ in this paper.
\begin{definition}
Let $H$ be a finite-dimensional Hopf algebra. The semidirect product (or, smash product) $H\rtimes\k\langle S^2\rangle$ of $H$ is defined through
\begin{itemize}
  \item $H\rtimes \k\langle S^2\rangle=H\otimes \k\langle S^2\rangle$ as a coalgebra;
  \item The multiplication is that $(h\rtimes S^{2i})(k\rtimes S^{2j}):=hS^{2i}(k)\rtimes S^{2(i+j)}$ for all $h,k\in H$ and $i,j\in\mathbb Z$;
  \item The unit element is $1\rtimes \id$;
  \item The antipode is then $S_{H\rtimes \k\langle S^2\rangle}:h\rtimes S^{2i}\mapsto S^{-2i+1}(h)\rtimes S^{-2i}$.
\end{itemize}
\end{definition}
Note that this is indeed a pivotal Hopf algebra (e.g. \cite[Theorem 2.13]{Mol77} and \cite[Proposition 2(1)]{Som98}) with a pivotal element $1\rtimes S^2$, and $H\cong H\rtimes\id \hookrightarrow H\rtimes\k\langle S^2\rangle$ is an inclusion of Hopf algebras.

The remaining of this section is devoted to establish a formula for $\exp(H\rtimes \Bbbk\langle S^2\rangle)$. For this purpose, following notation should be given, which could be regarded as special cases of \textit{twisted exponents} introduced in \protect{\cite[Definition 3.1]{SV17}} and \protect{\cite[Definition 3.1]{MVW16}}.

\begin{notation}\label{22}
Let $H$ be a finite-dimensional Hopf algebra. For any $i\in\mathbb Z$, we denote
$$\exp_{2i}(H):=\min\{n\geq 1\mid m_n\circ(\id\otimes S^{2i}\otimes\cdots\otimes S^{2(n-1)i})\circ\Delta_n=u\circ\varepsilon\}.$$
\end{notation}
Of course $\exp(H)$ is exactly $\exp_{-2}(H)$ with the notation.

The formula for $\exp(H\rtimes\k\langle S^2\rangle)$ would be established by steps.

\begin{lemma}\label{lem:QTexp}
Let $H$ be a finite-dimensional Hopf algebra. Then
\begin{itemize}
  \item[(1)] We have
        $$\exp(H\rtimes \k\langle S^2\rangle)=\lcm(\exp_{2i}(H)\mid i\in\mathbb Z);$$
  \item[(2)] If $H$ is quasitriangular, then for all $i\in\Z$,
        $$\exp_{4i}(H)=\exp_0(H)\;\;\;\;\text{and}\;\;\;\;\exp_{4i-2}(H)=\exp(H)$$
        hold, and thus $$\exp(H\rtimes \k\langle S^2\rangle)=\lcm(\exp_0(H),\exp(H)).$$
\end{itemize}
\end{lemma}

\begin{proof}
\begin{itemize}
  \item[(1)] This is \cite[Proposition 4.12]{LL??}, proved by direct computations.
  \item[(2)] Suppose that $H$ is quasitriangular. According to \cite[Section 3]{Dri89}, there exists a grouplike element $g\in H$ determining the inner automorphism $S^4$ on $H$. Thus Lemma \ref{lem:pivotpow} provides that
        $$(hg^i)^{[n]}=\sum h_{(1)}S^{4i}(h_{(2)})\cdots S^{4(n-1)i}(h_{(n)})g^{ni}
          \;\;\;\;(\forall i\in\Z)$$
        holds for all $h\in H$ and each $n\geq 1$. But it is clear that the (multiplication) order of $g$ divides $\gcd(\exp_0(H),\exp_{4i}(H))$. Thus the choice of $n$ being either $\exp_0(H)$ or $\exp_{4i}(H)$ in the equation shows that $n$ is divided by the other one. We conclude that $\exp_{4i}(H)=\exp_0(H)$ for each $i\in\Z$.

        Similarly, the equation $S^{4i-2}(h)=g^iS^{-2}(h)g^{-i}$ follows that
        \begin{eqnarray*}
        & & m_n\circ (\id\otimes S^{-2}\otimes\cdots\otimes S^{-2n+2})\circ\Delta_n(hg^i)  \\
        &=& \sum (h_{(1)}g^i)S^{-2}(h_{(2)}g^i)\cdots S^{-2n+2}(h_{(n)}g^i)  \\
        &=& \sum h_{(1)}\left(g^iS^{-2}(h_{(2)})g^{-i}\right)
                 \cdots \left(g^{(n-1)i}S^{-2n+2}(h_{(n)})g^{-(n-1)i}\right)\cdot g^{ni} \\
        &=& \sum h_{(1)}S^{4i-2}(h_{(2)})\cdots S^{(n-1)(4i-2)}(h_{(n)})g^{ni}\;\;\;\;(\forall i\in\Z)
        \end{eqnarray*}
        holds for all $h\in H$ and each $n\geq 1$. The (multiplication) order of $g$ also divides $\gcd(\exp(H),\exp_{4i-2}(H))$. Thus $\exp_{4i-2}(H)=\exp(H)$ for each $i\in\Z$ due to the same reason.
\end{itemize}
\end{proof}

\begin{corollary}\label{cor:Ddpivot}
Let $H$ be a finite-dimensional Hopf algebra. Then
$$\exp(D(H)\rtimes\k\langle S_{D(H)}{}^2\rangle)=\exp(H\rtimes\k\langle S^2\rangle).$$
\end{corollary}

\begin{proof}
It is clear that $\ord(S_{D(H)}{}^2)=\ord(S^2)$, and as a consequence
$\exp(H\rtimes\k\langle S^2\rangle)\mid\exp(D(H)\rtimes\k\langle S_{D(H)}{}^2\rangle)$
holds because of the inclusion of Hopf algebras:
\begin{eqnarray*}
H\rtimes \k\langle S^2\rangle &\hookrightarrow& D(H)\rtimes\k\langle S_{D(H)}{}^2\rangle \\
h\rtimes S^{2i} &\mapsto& (\varepsilon\bowtie h)\rtimes S_{D(H)}{}^{2i}.
\end{eqnarray*}

On the other hand, it is well-known that $D(H)$ is quasitriangular. Then according to the invariance under taking the Drinfeld double (Proposition \ref{prop:exp0basic2}(7) and Lemma \ref{lem:expbasic}(7)), 
as well as Lemma \ref{lem:QTexp}(2),
\begin{eqnarray*}
    \exp(D(H)\rtimes\k\langle S_{D(H)}{}^2\rangle)
&=& \lcm\left(\exp_0(D(H)),\exp(D(H))\right) \\
&=& \lcm\left(\exp_0(H),\exp(H)\right)
~\mid~ \exp(H\rtimes \k\langle S^2\rangle).
\end{eqnarray*}
The proof is then completed.
\end{proof}

\begin{corollary}\label{27}
Let $H$ be a finite-dimensional Hopf algebra. Then
$$\exp(H\rtimes\k\langle S^2\rangle)=\lcm(\exp_0(H),\exp(H)).$$
\end{corollary}

\begin{proof}
This is a conclusion of Lemma \ref{lem:QTexp}, Corollary \ref{cor:Ddpivot}, and the invariance under taking the Drinfeld double (Lemma \ref{lem:expbasic}(7) and Proposition \ref{prop:exp0basic2}):
\begin{eqnarray*}
    \exp(H\rtimes \k\langle S^2\rangle)
&=& \exp(D(H)\rtimes\k\langle S_{D(H)}{}^2\rangle) \\
&=& \lcm\left(\exp_0(D(H)),\exp(D(H))\right) \\
&=& \lcm\left(\exp_0(H),\exp(H)\right).
\end{eqnarray*}
\end{proof}

\begin{remark}
A similar result is \protect{\cite[Theorem 3.4]{MVW16}}. They describe the exponent of smash coproduct $H~\natural ~\k^G$ with $\exp(G)$ and twisted exponents of $H$, where $G$ acts as Hopf algebra automorphisms on the involutory Hopf algebra $H$.
\end{remark}

\section{Finiteness of Exponents}

In the final section, we study the finiteness (and upper bounds) of $\exp(H)$ and $\exp_0(H)$ for a finite-dimensional Hopf algebra $H$. Firstly recall in \cite[Theorem 4.3]{EG99} that $\exp(H)<\infty$ when $H$ is semisimple and cosemisimple, and $H$ must be involutory in this case (\cite[Theorem 3.1]{EG98}) which follows that $\exp_0(H)=\exp(H)<\infty$ holds. Therefore, we are supposed to focus on the case when $H$ is non-cosemisimple or non-semisimple.

Our results are mainly divided into two situations whether the characteristic of the base field $\k$ is $0$ or positive. Note that when $\mathrm{char}~\k=0$, the semisimplicity and cosemisimplicity for $H$ are equivalent (\cite[Theorem 3.3]{LR88}). Hence it is enough for us to consider the following two cases:
\begin{itemize}
  \item $H$ is non-cosemisimple in characteristic $0$;
  \item $H$ is finite-dimensional in positive characteristic.
\end{itemize}

\subsection{Finiteness of \texorpdfstring{$\exp_0(H)$}{exp0(H)}}

Recall in \protect{\cite[Theorem 4.10]{EG99}} that $\exp(H)<\infty$ as long as $H$ is finite-dimensional in positive characteristic. With the help of the semidirect product $H\rtimes\k\langle S^2\rangle$, we could directly infer that $\exp_0(H)$ is also finite in this case:

\begin{proposition}\label{26}
Let $H$ be a finite-dimensional Hopf algebra over a field $\k$ of positive characteristic. Then $\exp_0(H)<\infty$.
\end{proposition}

\begin{proof}
Since $H\rtimes \k\langle S^2\rangle$ is finite-dimensional and pivotal over $\k$, we know that
$$\exp_0(H\rtimes \k\langle S^2\rangle)=\exp(H\rtimes \Bbbk\langle S^2\rangle)<\infty.$$
However, $H\hookrightarrow H\rtimes \k\langle S^2\rangle$ is a Hopf subalgebra. Thus $\exp_0(H)<\infty$.
\end{proof}

When $H$ has the \textit{dual Chevalley property} (which means that its coradical $H_0$ is a Hopf subalgebra), we could discuss the finiteness of $\exp_0(H)$ more specifically. The notion of the \textit{coradical filtration} (cf. \cite[Section 9.1]{Swe69}) for $H$ would be mentioned, which is always denoted by $\{H_n\}_{n\geq 0}$ in this paper. We also recall that the \textit{Loewy length} (cf. \cite[Lemma 2.2]{Iov09}) of a finite-dimensional coalgebra $H$ is defined as
$$\mathrm{Lw}(H):=\min\{l\geq 1\mid H_{l-1}=H\}<\infty.$$

\begin{proposition}\label{prop:exp0}
Let $H$ be a finite-dimensional Hopf algebra with the dual Chevalley property over $\k$. Then
\begin{itemize}
  \item[(1)] If $H$ is non-cosemisimple and $\mathrm{char}~\k=0$, then $\exp_0(H)=\infty$;
  \item[(2)] If $\mathrm{char}~\k=p>0$, and denote $N:=\exp_0(H_0)<\infty$ and $L:=\mathrm{Lw}(H)$, then $\exp_0(H)\mid Np^M$, where $M$ is a positive integer satisfying $p^M\geq L$.
\end{itemize}
\end{proposition}

\begin{proof}
\begin{itemize}
  \item[(1)] This is \protect{\cite[Theorem 4.1]{LZ19}}.
  \item[(2)] Note that $N<\infty$ holds according to Proposition \ref{26}. The divisibility is followed directly from \protect{\cite[Lemma 4.11]{Shi15}}, which states that
      $$(id^{\ast N}-u\circ\varepsilon)^{\ast L}=0,$$
      where $\ast$ denotes the convolution in $\End_\k(H)$.
\end{itemize}
\end{proof}

\begin{remark}
Item (2) is generalizes \protect{\cite[Theorem 5.1]{LZ19}}, which holds when $H$ is pointed in positive characteristic. The author apologizes that the upper bound in \protect{\cite[Theorem 5.1]{LZ19}} is written incorrect, and the right version would be the same as Proposition \ref{prop:exp0}(2) here.
\end{remark}

\subsection{Primitive Matrices over Hopf Algebras}

In this subsection, we introduce the definitions and particular properties of so-called multiplicative and primitive matrices. They could be helpful when we deal with problems on exponents of non-pointed Hopf algebras. Related concepts are gathered in the following definition.

\begin{definition}(cf. \cite[Section 2.6]{Man88} and \cite[Section 3]{LZ19})
Let $(H,\Delta,\varepsilon)$ be a coalgebra.
\begin{itemize}
  \item[(1)] A matrix $\C=(c_{ij})_{r\times r}$ over $H$ is said to be multiplicative, if
      $$\Delta(c_{ij})=\sum\limits_{k=1}^r c_{ik}\otimes c_{kj}\;\;\;\;\text{and}\;\;\;\;
        \varepsilon(c_{ij})=\delta_{ij}$$
      hold for each $1\leq i,j\leq r$;
  \item[(2)] A multiplicative matrix $\C=(c_{ij})_{r\times r}$ is said to be basic, if $\{c_{ij}\mid 1\leq i,j\leq r\}$ is linearly independent;
  \item[(3)] Suppose that $\C=(c_{ij})_{r\times r}$ and $\mathcal{D}=(d_{ij})_{s\times s}$ are basic multiplicative matrices over $H$. A matrix $\X=(x_{ij})_{r\times s}$ over $H$ is said to be $(\C,\mathcal{D})$-primitive, if
      $$\Delta(x_{ij})=\sum\limits_{k=1}^r c_{ik}\otimes x_{kj}+\sum\limits_{l=1}^s x_{il}\otimes d_{lj}$$
      holds for each $1\leq i\leq r$ and $1\leq j\leq s$.
  \item[(4)] A primitive matrix $\X=(x_{ij})_{r\times s}$ over $H$ is said to be non-trivial, if some entry $x_{ij}$ does not belong to $H_0$.
\end{itemize}
\end{definition}

\begin{remark}
Suppose $\X$ is a $(\C,\mathcal{D})$-primitive matrix, it is immediate that for each entry $x$ in $\X$, we must have $\Delta(x)\in C\otimes H_1+H_1\otimes D$, where $C$ and $D$ are simple subcoalgebras with basic multiplicative matrices $\C$ and $\mathcal{D}$ respectively.
\end{remark}

Let $H$ be a Hopf algebra over $\k$ for the remaining of this paper. Before we discuss further properties and applications, some notations on matrices over $H$ should be given.

\begin{notation}
Let $\mathcal{A}=(a_{ij})_{r\times s}$ and $\mathcal{B}=(b_{ij})_{r'\times s'}$ be matrices over $H$. We give following notations:
\begin{itemize}
  \item[(1)] $\mathcal{A}^\mathrm{T}:=(a_{ji})_{s\times r}$ is the transpose of $\mathcal{A}$;
  \item[(2)] $\mathcal{A}+\mathcal{B}:=(a_{ij}+b_{ij})_{r\times s}$ when $r=r'$ and $s=s'$;
  \item[(3)] $h\mathcal{A}:=(ha_{ij})_{r\times s}$ and $\mathcal{A}h:=(a_{ij}h)_{r\times s}$ for any $h\in H$ (or $h\in \k$);
  \item[(4)] $\mathcal{A}\mathcal{B}:=\left(\sum\limits_{k=1}^s a_{ik}b_{kj}\right)_{r\times s'}$ when $s=r'$;
  \item[(5)] $f(\mathcal{A}):=(f(a_{ij}))_{r\times s}$ for any $\k$-linear map $f$ from $H$.
\end{itemize}
\end{notation}

As for later uses, we focus on existence and operation properties for certain non-trivial primitive matrices over $H$, especially when $H$ is non-cosemisimple with the dual Chevalley property. The set of all the simple subcoalgebras of $H$ is denoted by $\mathcal{S}$ for convenience.

\begin{lemma}\label{lem:X}
Let $H$ be a finite-dimensional non-cosemisimple Hopf algebra. Then:
\begin{itemize}
  \item[(1)] There exists a non-trivial $(\C,1)$-primitive matrix for some $C\in\S$ with a basic multiplicative matrix $\C$;
  \item[(2)] If $H_0$ has the dual Chevalley property, and suppose $\Lambda_0\in 1+\sum\limits_{D\in\S\setminus\{\k1\}}D$, then for any $C\in\S$ and non-trivial $(\C,1)$-primitive matrix $\X$, we have
      $$\Lambda_0\X\neq 0\;\;\;\;\text{and}\;\;\;\;\X\Lambda_0\neq0.$$
\end{itemize}
\end{lemma}

\begin{proof}
\begin{itemize}
\item[(1)]
This is a conclusion of \cite[Proposition 4.3]{LZ19} and \cite[Theorem 3.1]{LZ19}.
\item[(2)]
The dual Chevalley property of $H$ implies that $H_0H_1+H_1H_0\subseteq H_1$ (e.g. \cite[Lemma 5.2.8]{Mon93}). Let $x$ be an entry of $\X$ satisfying $x\notin H_0$. Since $\Delta(x)\in C\otimes H_1+H_1\otimes1$, it is evident that for any $D\in\S\setminus\{\k1\}$,
\begin{eqnarray*}
\Delta(Dx)~\subseteq~ DC\otimes DH_1+DH_1\otimes D
          ~\subseteq~ H_0\otimes H_1+H_1\otimes D,  \\
\Delta(xD)~\subseteq~ CD\otimes H_1D+H_1D\otimes D
          ~\subseteq~ H_0\otimes H_1+H_1\otimes D,
\end{eqnarray*}
both hold.

Now we choose a linear function $e:H\rightarrow\k$ such that
$$e(1)=1\;\;\;\;\text{while}\;\;\;\;e(D)=0\;\;(\forall D\in\S\setminus\{\k1\}).$$
It is true that $(\id\otimes e)\circ\Delta(x)\in Ce(H_1)+xe(1)\subseteq H_0+x$, because $\X$ is $(\C,1)$-primitive. Further computations show that
\begin{eqnarray*}
(\id\otimes e)\circ\Delta(\Lambda_0x)
&\in& (\id\otimes e)\circ\Delta(x+\sum\limits_{D\in\mathcal{S}\setminus\{\k1\}}Dx)  \\
&=& (\id\otimes e)\circ\Delta(x)
    +\sum\limits_{D\in\mathcal{S}\setminus\{\k1\}}(\id\otimes e)\circ\Delta(Dx)  \\
&\subseteq& (H_0+x)+H_0
~=~ x+H_0,
\end{eqnarray*}
and similarly $(\id\otimes e)\circ\Delta(x\Lambda_0)\in x+H_0$. However $x\notin H_0$, which follows that $\Lambda_0x$ and $x\Lambda_0$ are non-zero, then so are $\Lambda_0\X$ and $\X\Lambda_0$.
\end{itemize}
\end{proof}

The lemma above are actually enough for us to estimate $\exp_0(H)$ for a non-cosemisimple Hopf algebra $H$. However if we try to compute $\exp(H)$, we need to know how $S^2$ acts on primitive matrices. We provide follwing conclusions for this purpose.

\begin{lemma}(\cite[Proposition 3.7]{LL??})\label{lem:S2N}
Suppose $\k$ is algebraically closed. Let $H$ be a finite-dimensional Hopf algebra over $\k$ with the dual Chevalley property. Denote $N:=\exp(H_0)<\infty$. Then for any basic multiplicative matrix $\C$ and $(\C,1)$-primitive matrix $\X$, we have $S^{2N}(\X)=\X$.
\end{lemma}

\begin{remark}
The fact that $\exp(H_0)<\infty$ is due to a discussion on the characteristic of $\k$ as well as the fact that $H_0$ is a semisimple Hopf algebra. See \cite[Corollary 2.4]{LL??} for example.
\end{remark}

\begin{corollary}\label{cor:S2X}
Let $\k$ be an algebraically closed field of characteristic $0$. Suppose $H$ is a finite-dimensional non-cosemisimple Hopf algebra with the dual Chevalley property over $\k$. Denote $N:=\exp(H_0)$. Then there exists a non-trivial $(\C,1)$-primitive matrix $\X$ for some basic multiplicative matrix $\C$, such that
$$S^2(\X)=q\X,$$
where $q\in\k$ is a $N$th root of unity.
\end{corollary}

\begin{proof}
We know by Lemma \ref{lem:X}(1) that there do exist non-trivial $(\C,1)$-primitive matrices for some basic multiplicative matrix $\C$. Let $P\neq 0$ be the finite-dimensional space of all the $(\C,1)$-primitive matrices over $H$.

Note that the dual Chevalley property ensures that $H_0$ is involutory in characteristic $0$, and thus $P$ is stable under $S^2$. Now we consider the representation of the cyclic group $\Z_N$ on $P$ defined by
$$n:\mathcal{W}\mapsto S^{2n}(\mathcal{W})\;\;\;\;(n\in\Z_N,\;\mathcal{W}\in P),$$
which is well defined due to Lemma \ref{lem:S2N}.

Clearly this is a direct sum of $1$-dimensional representations. We claim that one of these subrepresentations must have non-trivial basis $\X$, otherwise there would be no non-trivial $(\C,1)$-primitive matrices in $P$, a contradiction. Of course the subrepresentation $\k\X$ provides that $S^2(\X)=q\X$, where $q$ is an $N$th root of unity.
\end{proof}

\subsection{Finiteness with the Dual Chevalley Property}

The finiteness of $\exp(H)$ is discussed at the end of this paper. When $H$ is non-semisimple in characteristic $0$, it is estimated that $\exp(H)$ is usually (and probably always) infinite in \cite[Section 1]{EG02}. We show that this property is true as long as $H$ has the dual Chevalley property.

\begin{theorem}
Let $H$ be a finite-dimensional Hopf algebra with the dual Chevalley property over $\k$. Then
\begin{itemize}
  \item[(1)] If $H$ is non-cosemisimple and $\mathrm{char}~\k=0$, then $\exp(H)=\infty$;
  \item[(2)] If $\mathrm{char}~\Bbbk=p>0$, and denote $N:=\lcm(\exp(H_0),\exp_0(H_0))<\infty$ and $L:=\mathrm{Lw}(H)$, then $\exp(H)\mid Np^M$, where $M$ is a positive integer satisfying $p^M\geq L$.
\end{itemize}
\end{theorem}

\begin{proof}
\begin{itemize}
\item[(1)]
Without the loss of generality, $\k$ is assumed to be algebraically closed. We begin with the non-cosemisimplicity of $H$. According to Corollary \ref{cor:S2X}, there must be a non-trivial $(\C,1)$-primitive matrix $\X$ for some basic multiplicative matrix $\C$, satisfying $S^2(\X)=q\X\;(0\neq q\in\k)$. On the other hand, the following equations about $\X$ could be obtained by direct computations (see \cite[Lemma 3.5]{LL??}):
\begin{equation}\label{SandS2}
S(\X)=-S(\C)\X\;\;\;\;\text{and}\;\;\;\;q\X=S^2(\X)=((S(\C)\X)^\T S^2(\C)^\T)^\T.
\end{equation}
Note that $H_0$ is involutory, which follows that $S^2(\C)=\C$ as well as $\C^\T S(\C)^\T= S(\C)^\T \C^\T=I$ (the identity matrix over $H$). Thus the latter one within Equation (\ref{SandS2}) could be changed to
\begin{equation}\label{commuting}
S(\C)\X=q(\X^\T S(\C)^\T)^\T.
\end{equation}

Now we focus on showing that
\begin{eqnarray}\label{Xpower2}
m_n\circ(\id\otimes S^2\otimes\cdots\otimes S^{2n-2})\circ\Delta_n(\X)
&=& (I+q\C+\cdots q^{n-1}\C^{n-1})\X  \\
&=& \sum\limits_{i=0}^{n-1}q^i\C^i\X  \nonumber
\end{eqnarray}
is non-zero for any $n\geq 1$, which would imply that $\exp_{2}(H)=\infty$ because of $\varepsilon(\X)=0$. Denote the integral of $H_0$ by $\Lambda_0$ which belongs in $1+\sum\limits_{D\in\S\setminus\{\k1\}}D$, and compute firstly that for each $i\geq 0$,
\begin{eqnarray*}
&& q^{-i}S(\C)^i\X\Lambda_0  \\
&=& q^{-i}S(\C)^{i-1}(S(\C)\X)\Lambda_0
\;\;=\;\; q^{-i}S(\C)^{i-1}(q^{-1}(\X^\T S(\C)^\T)^\T)\Lambda_0  \\
&=& q^{-i+1}S(\C)^{i-1}(\X^\T S(\C)^\T\Lambda_0)^\T \;\;=\;\; q^{-i+1}S(\C)^{i-1}(\X^\T\Lambda_0)^\T  \\
&=& q^{-i+1}S(\C)^{i-1}\X\Lambda_0 \;\;=\;\; \cdots  \\
&=& q^{-1}S(\C)\X\Lambda_0 \;\;=\;\; \X\Lambda_0
\end{eqnarray*}
holds, where the second equality is due to Equation (\ref{commuting}). As a consequence, for any $n\geq1$,
\begin{eqnarray*}
q^{-n+1}S(\C)^{n-1}\left(\sum\limits_{i=0}^{n-1}q^i\C^i\X\right)\Lambda_0
&=& \sum\limits_{i=0}^{n-1}q^{-(n-1-i)}S(\C)^{n-1-i}\X\Lambda_0  \\
&=& \sum\limits_{i=0}^{n-1}q^{-i}S(\C)^i\X\Lambda_0
\;\;=\;\; n\X\Lambda_0 \neq 0
\end{eqnarray*}
according to Lemma \ref{lem:X}(2) and $\mathrm{char}~\k=0$. This clearly follows that (\ref{Xpower2}) must be non-zero for any positive integer $n$, and thus $\exp_2(H)=\infty$.

Finally, it follows from Lemma \ref{lem:QTexp}(2) that
$$\exp(H)=\exp(D(H))=\exp_2(D(H))\geq\exp_2(H)=\infty.$$
Note that the inequality holds since $H$ is a Hopf subalgebra of its Drinfeld double $D(H)$.

\item[(2)] It is known that $H\rtimes\k\langle S^2\rangle$ has the dual Chevalley property and the Lowey length $L$ when $H$ does so. However, its coradical is pivotal and has exponent
    \begin{eqnarray*}
    \exp(H_0\rtimes\k\langle S^2\rangle) &=& \exp_0(H_0\rtimes\k\langle S^2\rangle)  \\
    &=& \lcm(\exp(H_0),\exp_0(H_0))\;\;=\;\; N.
    \end{eqnarray*}
    According to Proposition \ref{prop:exp0}(2), we find that $\exp_0(H\rtimes\k\langle S^2\rangle)$ divides $Np^L$. Consequently
    $$\exp(H)\mid Np^L$$
    holds as well, since $\exp(H)$ divides $\exp(H\rtimes\k\langle S^2\rangle)$ which equals to $\exp_0(H\rtimes\k\langle S^2\rangle)$.
\end{itemize}
\end{proof}


\section*{Acknowledgement}

The author would like to thank Professor Gongxiang Liu for his useful supports and discussions.

\end{document}